\documentclass[12pt,a4paper]{amsart}
\usepackage{amsfonts}
\usepackage{amsthm}
\usepackage{amsmath}
\usepackage{amscd}
\usepackage[latin2]{inputenc}
\usepackage{t1enc}
\usepackage[mathscr]{eucal}
\usepackage{indentfirst}
\usepackage{graphicx}
\usepackage{graphics}
\numberwithin{equation}{section}
\usepackage[margin=2.9cm]{geometry}
\usepackage{epstopdf}

\theoremstyle{plain}
\newtheorem{Th}{Theorem}[section]
\newtheorem{Lemma}[Th]{Lemma}
\newtheorem{Cor}[Th]{Corollary}

 \theoremstyle{definition}
\newtheorem{defi}[Th]{Definition}

\newtheorem{?}[Th]{Problem}

\begin{document}

\title{Sasakian Manifold and  $*$-Ricci Tensor}

\author{Venkatesha $\cdot$ Aruna Kumara H}

\address{Department of Mathematics, Kuvempu University,\\
	Shankaraghatta - 577 451, Shimoga, Karnataka, INDIA.\\} 

\email{vensmath@gmail.com, arunmathsku@gmail.com}

\begin{abstract}  The purpose of this paper is to study $*$-Ricci tensor on Sasakian manifold. Here, $\varphi$-confomally flat and confomally flat $*$-$\eta$-Einstein Sasakian manifold are studied. Next, we consider $*$-Ricci symmetric conditon on Sasakian manifold. Finally, we study a special type of metric called $*$-Ricci soliton on Sasakian manifold.
\end{abstract}

\maketitle

\section{Introduction} The notion of contact geometry has evolved from the mathematical
formalism of classical mechanics\cite{Gei}. Two important classes of contact manifolds are $K$-contact manifolds and Sasakian manifolds\cite{DEB}. An odd dimensional analogue of Kaehler geometry is the Sasakian geometry. Sasakian manifolds were firstly studied by the famous geometer Sasaki\cite{Sasa} in 1960, and for long time focused on this, Sasakian manifold have been extensively studied under several points of view in \cite{Chaki,De,Ikw,Olz,Tanno}, and references therein.
\par On the other hand, it is mentioned that  the notion
of $*$-Ricci tensor was first introduced by Tachibana \cite{TS} on almost
Hermitian manifolds and further studied by Hamada and Inoguchi \cite{HT} on real hypersurfaces of non-flat complex space forms.
\par Motivated by these studies the present paper is organized as follows: In section 2, we recall some basic formula and result concerning Sasakian manifold and $*$-Ricci tensor which we will use in further sections. A $\varphi$-conformally flat Sasakian manifold is studied in section 3, in which we obtain some intersenting result. Section 4 is devoted to the study of conformally flat $*$-$\eta$-Einstein Sasakian manifold. In section 5, we consider $*$-Ricci symmetric Sasakian manifold and found that  $*$-Ricci symmetric Sasakian manifold is $*$-Ricci flat, moreover, it is $\eta$-Einstein manifold. In the last section, we studied a special type of metric called $*$-Ricci soliton. Here we have proved a important result on Sasakian manifold admitting $*$-Ricci soliton. 

\section{Preliminaries}%----------------------------------------------------------------------------------------------------------------
In this section, we collect some general definition and basic formulas on contact metric manifolds and Sasakian manifolds which we will use in further sections. We may refer to \cite{Ar,Bo,Kus} and references therein for more details and information about Sasakian geometry.
\par A (2n+1)-dimensional smooth connected manifold $M$ is called almost contact manifold if it admits a triple $(\varphi, \xi, \eta)$, where $\varphi$ is a tensor field of type $(1,1)$, $\xi$ is a global vector field and $\eta$ is a 1-form, such that 
\begin{align}
\label{2.1} \varphi^2X=-X+\eta(X)\xi, \qquad \eta(\xi)=1, \qquad \varphi\xi=0, \qquad \eta\circ\varphi=0,
\end{align}  
for all $X, Y\in TM$. If an almost contact manifold $M$ admits a structure $(\varphi, \xi, \eta, g)$, $g$ being a Riemannian metric such that
\begin{align}
\label{2.2} g(\varphi X,\varphi Y)=g(X,Y)-\eta(X)\eta(Y),
\end{align} 
then $M$ is called an almost contact metric manifold. An almost contact metric manifold $M(\varphi, \xi, \eta, g)$ with $d\eta(X,Y)=\Phi(X,Y)$, $\Phi$ being the fundamental 2-form of $M(\varphi, \xi,\eta,g)$ as defined by $\Phi(X,Y)=g(X,\varphi Y)$, is a contact metric manifold and $g$ is the associated metric. If, in addition, $\xi$ is a killing vector field (equivalentely, $h=\frac{1}{2}L_\xi \varphi=0$, where $L$ denotes Lie differentiation), then the manifold is called $K$-contact manifold. It is well known that\cite{DEB}, if the contact metric structure $(\varphi,\xi,\eta,g)$ is normal, that is, $[\varphi,\varphi]+2d\eta\otimes\xi=0$ holds, then $(\varphi,\xi,\eta,g)$ is Sasakian. An almost contact metric manifold is Sasakian if and only if
\begin{align}
\label{2.3} (\nabla_X \varphi)Y=g(X,Y)\xi-\eta(Y)X,
\end{align}
for any vector fields $X, Y$ on $M$, where $\nabla$ is Levi-Civita connection of $g$. A Sasakian manifold is always a $K$-contact manifold. The converse also holds when the dimension is three, but which may not be true in higher dimensions\cite{JB}. On Sasakian manifold, the following relations are well known;
\begin{align}
\label{a2.4}\nabla_X \xi&=-\varphi X\\
\label{2.4} R(X,Y)\xi&=\eta(Y)X-\eta(X)Y,\\
\label{2.5} R(\xi, X)Y&=g(X,Y)\xi-\eta(Y)X,\\
\label{2.6} Ric(X,\xi)&=2n\eta(X)\qquad (or\,\, Q\xi=2n\xi),
\end{align}
for all $X,Y\in TM$, where $R, Ric$ and $Q$ denotes the curvature tensor, Ricci tensor and Ricci operator, respectively.
\par On the other hand, let $M(\varphi,\xi,\eta,g)$ be an almost contact metric manifold with Ricci tensor $Ric$. The $*$-Ricci tensor and $*$-scalar curvature of $M$ repectively are defined by
\begin{align}
\label{2.7}Ric^*(X,Y)=\sum_{i=1}^{2n+1}R(X,e_i,\varphi e_i, \varphi Y),\qquad r^*=\sum_{i=1}^{2n+1}Ric^*(e_i,e_i),
\end{align}
for all $X,Y\in TM$, where ${e_1,...,e_{2n+1}}$ is an orthonormal basis of the tangent space $TM$. By using the first Bianchi identity and \eqref{2.7} we get
\begin{align}
Ric^*(X,Y)=\frac{1}{2}\sum_{i=1}^{2n+1}g(\varphi R(X,\varphi Y)e_i,e_i).
\end{align} 
An almost contact metric manifold is said to be $*$-Einstein if $Ric^*$ is a constant
multiple of the metric $g$. One can see $Ric^*(X,\xi)=0$, for all $X \in TM$. It should be remarked that $Ric^*$ is not symmetric, in general. Thus the
condition $*$-Einstein automatically requires a symmetric property of the $*$-Ricci tensor\cite{HT}.

Now we make an effort to find $*$-Ricci tensor on Sasakian manifold.
\begin{Lemma}
	In a (2n+1)-dimensional Sasakian manifold $M$, the $*$-Ricci tensor is given by
	\begin{align}
	\label{2.9} Ric^*(X,Y)=Ric(X,Y)-(2n-1)g(X,Y)-\eta(X)\eta(Y).
	\end{align}
\end{Lemma}
\begin{proof}
	In a (2n+1)-dimensional Sasakian manifold $M$, the Ricci tensor $Ric$ satisfies the relation (see page 284, Lemma 5.3 in \cite{YKK}):
	\begin{align}
	\label{2.10}	Ric(X,Y)=\frac{1}{2}\sum_{i=1}^{2n+1}g(\varphi R(X,\varphi Y)e_i,e_i)+(2n-1)g(X,Y)+\eta(X)\eta(Y).
	\end{align}
	Using the definition of $Ric^*$ in \eqref{2.10}, we obtain \eqref{2.9}
\end{proof}	
\begin{defi}\label{d2.1} \cite{JTC} An almost contact metric manifold $M$ is said to be weakly $\varphi$-Einstein if
	\begin{align*}
	Ric^\varphi(X,Y)=\beta g^\varphi(X,Y),\quad X,Y\in TM,
	\end{align*}
	for some function $\beta$. Here $Ric^\varphi$ denotes the symmetric part of $Ric^*$, that is,
	\begin{align*}
	Ric^\varphi(X,Y)=\frac{1}{2}\{Ric^*(X,Y)+Ric^*(Y,X)\}, \quad X,Y\in TM,
	\end{align*}
	we call $Ric^\varphi$, the $\varphi$-Ricci tensor on $M$ and the symmetric tensor $g^\varphi$ is defined by $g^\varphi(X,Y)=g(\varphi X, \varphi Y)$. When $\beta$ is constant, then $M$ is said to be $\varphi$-Einstein.
\end{defi}
\begin{defi}
	If the Ricci tensor of a Sasakian manifold $M$ is of the form
	\begin{align*}
	Ric(X,Y)=\alpha g(X,Y)+\gamma \eta(X)\eta(Y),
	\end{align*} 
	for any vector fields $X, Y$ on $M$, where $\alpha$ and $\gamma$ being constants, then $M$ is called an $\eta$-Einstein manifold. 
\end{defi}
Let $M(\varphi,\xi,\eta,g)$ be a Sasakian $\eta$-Einstein manifold with constants $(\alpha, \gamma)$. Consider a $D$-homothetic Sasakian structure $S = (\varphi',\xi',\eta',g')=(\varphi,a^{-1}\xi,a\eta,ag + a(a-1)\eta\otimes\eta)$.  Then $(M,S)$ is also $\eta$-Einstein with constants
$\alpha'=\frac{\alpha+2-2a}{a}$ and $\gamma'=2n-\alpha'$ (see proposition 18 in \cite{Cha}).		
Here we make a remark that the particular value: $\alpha=-2$ remains fixed under a
$D$-homothetic deformation\cite{AG}. Thus, we state the following definition.
\begin{defi}
	A Sasakian $\eta$-Einstein manifold with $\alpha=-2$ is said to be
	D-homothetically fixed.
\end{defi}

\section{$\varphi$-conformally Flat Sasakian Manifold}
The Weyl conformal curvature tensor\cite{YKK} is defined as a map $C:TM\times TM\times TM \longrightarrow TM$ such that
\begin{align}
\nonumber C(X,Y)Z=&R(X,Y)Z-\frac{1}{2n-1}\{Ric(Y,Z)X-Ric(X,Z)Y+g(Y,Z)QX\\
\label{3.1} &-g(X,Z)QY\}+\frac{r}{2n(2n-1)}\{g(Y,Z)X-g(X,Z)Y\}, \quad X,Y\in TM.
\end{align}
In\cite{CAR}, Cabrerizo et al proved some necessary condition for $K$-contact manifold to be $\varphi$-conformally flat. In the following theorem we find a condition for $\varphi$-conformally flat Sasakian manifold.
\begin{Th}
	If a (2n+1)-dimensional Sasakian manifold $M$ is $\varphi$-conformally flat, then $M$ is $*$-$\eta$-Einstein manifold. Moreover, $M$ is weakly $\varphi$-Einstein.
\end{Th}
\begin{proof}
	It is well known that (see in \cite{CAR}), if a $K$-contact manifold is $\varphi$-conformally flat then we get the following relation:
	\begin{align}
	\label{3.2} R(\varphi X, \varphi Y, \varphi Z, \varphi W)=\frac{r-4n}{2n(2n-1)}\{g(\varphi Y,\varphi Z)g(\varphi X, \varphi W)-g(\varphi X,\varphi Z)g(\varphi Y, \varphi W)\}.
	\end{align}
	In a Sasakian manifold, in view of \eqref{2.4} and \eqref{2.5} we can verify that
	\begin{align}
	\nonumber R(\varphi^2 X, \varphi^2Y,\varphi^2Z,\varphi^2W)=&R(X,Y,Z,W)-g(Y,Z)\eta(X)\eta(Y)+g(X,Z)\eta(Y)\eta(W)\\
	\label{3.3} &+g(Y,W)\eta(X)\eta(Z)-g(X,W)\eta(Y)\eta(Z),
	\end{align}	
	for all $X,Y,Z,W \in TM$. Replacing $X, Y,Z,W$ by $\varphi X, \varphi Y, \varphi Z, \varphi W$ respectively in \eqref{3.2} and making use of \eqref{2.2} and \eqref{3.3} we get
	\begin{align}
	\nonumber R(X,Y,Z,W)&=\frac{r-4n}{2n(2n-1)}\{g(Y,Z)g(X,W)-g(X,Z)g(Y,W)\}\\
	\nonumber&-\frac{r-2n(2n+1)}{2n(2n-1)}\{g(Y,Z)\eta(X)\eta(W)-g(X,Z)\eta(Y)\eta(W)\\
	&+g(X,W)\eta(Y)\eta(Z)-g(Y,W)\eta(X)\eta(Z)\}.
	\end{align}
	By the definition of $Ric^*$, direct computation yields
	\begin{align}
	\label{3.5} Ric^*(X,Y)=\sum_{i=1}^{2n+1}R(X,e_i,\varphi e_i, \varphi Y)=\beta g(X,Y)-\beta\eta(X)\eta(Y),
	\end{align}
	where $\beta=\frac{r-4n}{2n(2n-1)}$, showing that $M$ is $*$-$\eta$-Einstein. Next, in view of \eqref{2.2} we have 
	\begin{align}
	\label{3.6} Ric^*(X,Y)=\frac{r-4n}{2n(2n-1)} g^\varphi(X,Y),
	\end{align}
	for all $X,Y \in TM$. Hence $Ric^*=Ric^\varphi$ and hence it is weakly $\varphi$-Einstein. This completes the proof.
\end{proof}	
Suppose the scalar curvature of the manifold is constant. Then in view of \eqref{3.6}, we have
\begin{Cor}
	If a $\varphi$-conformally flat Sasakian manifold has constant scalar curvature, then it is $\varphi$-Einstein.
\end{Cor}

In a Sasakian manifold, the $*$-Ricci tensor is given by \eqref{2.9} and so in view of \eqref{3.5}, we state the following;
\begin{Cor}
	A $\varphi$-conformally flat Sasakian manifold is  $\eta$-Einstein.
\end{Cor}\
The notion of $\eta$-parallel Ricci tensor was introduced in the context of Sasakian manifold by Kon\cite{KON} and is defined by $(\nabla_Z Ric)(\varphi X,\varphi Y)=0$, for all $X,Y \in TM$. From this definition, we define a $\eta$-parallel $*$-Ricci tensor by $(\nabla_Z Ric^*)(\varphi X,\varphi Y)=0$.

Replacing $X$ by $\varphi X$ and $Y$ by $\varphi Y$ in \eqref{3.5}, we obtain
$Ric^*(\varphi X,\varphi Y)=\beta g(\varphi X,\varphi Y)$.
Now taking covariant differention with respect to $W$, we get $(\nabla_W Ric^*)(\varphi X,\varphi Y)=dr(W) g(\varphi X, \varphi Y)$. Therefore we have the following;
\begin{Cor}
	A (2n+1)-dimensional $\varphi$-conformally flat Sasakian manifold has $\eta$-parallel $*$-Ricci tensor if and only if the scalar curvature of the
	manifold is constant.
\end{Cor}

\section{Conformally Flat $*$-$\eta$-Einstein Sasakian Manifold}
Suppose $M$ is conformally flat Sasakian manifold, then from \eqref{3.1} we have
\begin{align}
\nonumber R(X.Y)Z=&\frac{1}{2n-1}\{Ric(Y,Z)X-Ric(X,Z)Y+g(Y,Z)QX\\
\label{4.1} &-g(X,Z)QY\}-\frac{r}{2n(2n-1)}\{g(Y,Z)X-g(X,Z)Y\}.
\end{align}
If we set $Y=Z=\xi\perp X$, we find $QX=\frac{r-2n}{2n}X$. From this equation, \eqref{4.1} becomes
\begin{align}
\label{4.2} R(X,Y)Z=\frac{r-4n}{2n(2n-1)}\{g(Y,Z)X-g(X,Z)Y\}.
\end{align}
By definition of $*$-Ricci tensor, direct computation yields
\begin{align}
\label{4.3} Ric^*(X,Y)=\frac{r-4n}{2n(2n-1)}g(\varphi X, \varphi Y).
\end{align}
Since M is a conformally flat Sasakian manifold, we have
the following equations from the definition of $Ric^*$ and equation \eqref{4.2}:
\begin{align*}
Ric^*(\varphi Y, \varphi X)&=\sum_{i=1}^{2n+1}R(\varphi Y, e_i, \varphi e_i, \varphi^2X)\\
&=\sum_{i=1}^{2n+1}\{-R(X,\varphi e_i,e_i,\varphi Y)+\eta(X)R(\varphi Y,e_i,\varphi e_i,\xi)\}\\
&=Ric^*(X,Y).
\end{align*}
If we set $Y=\varphi X$ such that $X$ is unit, we obatin $Ric^*(\varphi^2X,\varphi X)=Ric^*(X,\varphi X)$ which implies that $Ric^*(X,\varphi X)=0$. Thus, from the definition of $*$-$\eta$-Einstein and \eqref{4.3}, we obtain $a g(X,Y)+b \eta(X)\eta(Y)=\frac{r-4n}{2n(2n-1)}g(\varphi X,\varphi Y)$. If we choose $X=Y=\xi$, we find $a+b=0$. If we set $Y=X\perp \xi$ such that $X$ and $Y$ are units, we get
\begin{align}
\label{4.4} a=\frac{r-4n}{2n(2n-1)}=K(X,\varphi X).
\end{align} 
In\cite{OKU}, the author proved that every conformally flat Sasakian manifold has a constant curvature +1, that is, $R(X,Y)Z=g(Y,Z)X-g(X,Z)Y$. From this result and \eqref{4.2}, we find $r=2n(2n-1)+4n$. In view of \eqref{4.4}, we obtain $a=1$. Therefore we have the following:
\begin{Th}\label{t4.2}
	Let $M$ be a (2n+1)-dimensional conformally flat Sasakian manifold. If $M$ is $*$-$\eta$-Einstein, then it is of constant curvature +1.
\end{Th}
We know that every Riemannian manifold of constant sectional curvature is locally symmetric. From theorem \eqref{t4.2}, we have
\begin{Cor}
	A conformally flat $*$-$\eta$-Einstein Sasakian manifold is locally symmetric.
\end{Cor}
\begin{Th}
	A (2n+1)-dimensional conformally flat Sasakian manifold is $\varphi$-Einstein. 
\end{Th}
\begin{proof}
	The theorem follows from \eqref{4.3} and definition \eqref{d2.1}.
\end{proof}	

\section{$*$-Ricci Semi-symmetric Sasakian Manifold}
A contact metric manifold is called Ricci semi-symmetric if $R(X,Y)\cdot Ric=0$, for all $X,Y \in TM$. Analogous to this definition, we define $*$-Ricci semi-symmetric by $R(X,Y)\cdot Ric^*=0$.
\begin{Th}
	If a (2n+1)-dimensional Sasakian manifold $M$ is $*$-Ricci semi-symmetric, then $M$ is $*$-Ricci flat. Moreover, it is $\eta$-Einstein manifold and
	the Ricci tensor can be expressed as
	\begin{align*}
	Ric(X,Y)=(2n-1)g(X,Y)+\eta(X)\eta(Y).
	\end{align*}
\end{Th}
\begin{proof}
	Let us consider (2n+1)-dimensional Sasakian manifold which satisfies the condition $R(X,Y).Ric^*=0$. Then we have
	\begin{align}
	\label{5.1} Ric^*(R(X,Y)Z,W)+Ric^*(Z,R(X,Y)W)=0.
	\end{align}	
	Putting $X=Z=\xi$ in \eqref{5.1}, we have
	\begin{align}
	\label{5.2} Ric^*(R(\xi, Y)\xi,W)+Ric^*(\xi,R(\xi,Y)W)=0.
	\end{align}	
	It is well known that $Ric^*(X,\xi)=0$. Making use of \eqref{2.4} in \eqref{5.2} and by virtue of last equation, we find
	\begin{align}
	\label{5.3} Ric^*(Y,W)=0, \qquad Y,W\in TM,
	\end{align}	
	showing that $M$ is $*$-Ricci flat. Moreover, in view of \eqref{5.3} and \eqref{2.9}, we have the required result.
\end{proof}

\section{Sasakian Manifold Admitting $*$-Ricci Soliton}

%--------------------------------------------------------------------------  
\par Ricci flows are intrinsic geometric flows on a Riemannian manifold,
whose fixed points are solitons and it was introduced by Hamilton\cite{HRS}. Ricci solitons also correspond to self-similar solutions of Hamilton's Ricci 
flow. They are natural generalization of Einstein metrics and is defined by
\begin{align}
\label{R1}(L_Vg)(X,Y)+2Ric(X,Y)+2\lambda g(X,Y)=0,
\end{align}
for some constant $\lambda$, a potential vector field $V$. The Ricci soliton is said to be shrinking, steady, and expanding according as $\lambda$ is negative, zero, and positive respectively.
\par The notion of $*$-Ricci soliton was introduced by George and Konstantina\cite{Geo}, in 2014, where they essentially modified the definition of Ricci soliton by replacing the Ricci tensor in \eqref{R1} with the $*$-Ricci tensor. Recently, the authors studied $*$-Ricci soliton on para-Sasakian manifold in the paper \cite{Prak} and obtain several interesting result. A Riemannian metric $g$ on $M$ is called $*$-Ricci soliton, if there is vector field $V$, such that
\begin{align}
\label{R2} (L_Vg)(X,Y)+2Ric^*(X,Y)+2\lambda g(X.Y)=0,
\end{align}
for all vector fields $X,Y$ on $M$. In this section we study a special type of metric called $*$-Ricci soliton on Sasakian manifold. Now we prove the following result:
\begin{Th}
	If the metric $g$ of a (2n+1)-dimensional Sasakian manifold $M(\varphi,\xi,\eta,g)$
	is a $*$-Ricci soliton with potential vector field $V$, then (i) $V$ is Jacobi along geodesic of $\xi$, (ii) M is an $\eta$-Einstein manifold and the Ricci tensor can be expressed as
	\begin{align}
	\label{6.1}	Ric(X,Y)=\left[2n-1-\frac{\lambda}{2}\right]g(X,Y)+\left[1+\frac{\lambda}{2}\right]\eta(X)\eta(Y).
	\end{align}
	
\end{Th}
\begin{proof}
	Call the following commutation formula (see \cite{YK}, page 23): 	
	\begin{align}
	(L_V\nabla_X g-\nabla_X L_V g-\nabla_{[V,X]}g)(Y,Z)=-g((L_V\nabla)(X,Y),Z)-g((L_V \nabla)(X,Z),Y),
	\end{align}
	and is well known for all vector fields $X,Y,Z$ on $M$. Since $g$ is parallel with respect to Levi-Civita connection $\nabla$, then the above relation becomes
	\begin{align}
	\label{6.2}(\nabla_XL_V g)(Y,Z)=g((L_V \nabla)(X,Y),Z)+g((L_V \nabla)(X,Y),Z).
	\end{align} 
	Since $L_V\nabla$ is a symmetric tensor of type $(1,2)$, i.e., $(L_V \nabla)(X,Y)=(L_V\nabla)(Y,Z)$, it follows from \eqref{6.2} that
	\begin{align}
	\label{6.3}g((L_V\nabla)(X,Y),Z)=\frac{1}{2}\{(\nabla_XL_V g)(Y,Z)+(\nabla_YL_V g)(Z,X)-(\nabla_ZL_Vg)(X,Y)\}.
	\end{align}
	Next, taking covariant differentiation of $*$-Ricci soliton equation \eqref{R2} along a vector field $X$, we obatin $(\nabla_X L_V g)(X,Y)=-2(\nabla_X Ric^*)(X,Y)$. Substitutimg this relation into \eqref{6.3} we have
	\begin{align}
	\label{6.4} g((L_V\nabla)(X,Y),Z)=(\nabla_Z Ric^*)(X,Y)-(\nabla_X Ric^*)(Y,Z)-(\nabla_Y Ric^*)(X,Z).
	\end{align}
	Again, taking covariant differentiation of \eqref{2.9} with respect to $Z$, we get
	\begin{align}
	\label{6.5} (\nabla_Z Ric^*)(X,Y)=(\nabla_Z Ric)(X,Y)-\{g(Z,\varphi X)\eta(Y)+g(Z,\varphi Y)\eta(X)\}.
	\end{align}
	Combining \eqref{6.5} with \eqref{6.4}, we find
	\begin{align}
	\nonumber g((L_V\nabla)(X,Y),Z)=&(\nabla_Z Ric)(X,Y)-(\nabla_X Ric)(Y,Z)-(\nabla_Y Ric)(X,Z)\\
	\label{6.6} &+2g(X,\varphi Z)\eta(Y)+2g(Y,\varphi Z)\eta(X).
	\end{align}
	In Sasakian manifold we know the following relation\cite{AG}:
	\begin{align}
	\label{6.7} \nabla_\xi Q=Q\varphi-\varphi Q=0, \qquad (\nabla_X Q)\xi=Q\varphi X-2n\varphi X.
	\end{align}
	Replacing $Y$ by $\xi$ in \eqref{6.6} and then using \eqref{6.7} we obtain
	\begin{align}
	\label{6.8} (L_V \nabla)(X,\xi)=-2Q\varphi X+2(2n-1)\varphi X.
	\end{align}
	From the above equation, we have
	\begin{align}
	\label{6.9} (L_V\nabla)(\xi,\xi)=0.
	\end{align}
	Now, substituting $X=Y=\xi$ in the well known formula \cite{YK}:
	\begin{align*}
	(L_V \nabla)(X,Y)=\nabla_X\nabla_Y V-\nabla_{\nabla_X Y}V+R(V,X)Y,
	\end{align*}
	and then making use of equation \eqref{6.9} we obtain
	\begin{align*}
	\nabla_\xi\nabla_\xi V+R(V,\xi)\xi=0,
	\end{align*}
	which proves part (i).
	Further, differentiating \eqref{6.8} covariantly along an arbitrary vector field $Y$ on $M$ and then using \eqref{2.3} and last equation of \eqref{2.6}, we obtain
	\begin{align}
	\nonumber &(\nabla_Y L_V \nabla)(X,\xi)-(L_V\nabla)(X,\varphi Y)\\
	\label{6.10} &=2\{-(\nabla_Y Q)\varphi X-g(X,Y)\xi+\eta(X)QY-(2n-1)\eta(X)Y\}.
	\end{align}
	According to Yano\cite{YK}, we have the following commutation formula:
	\begin{align}
	\label{c1} (L_V R)(X,Y)Z=(\nabla_X L_V \nabla)(Y,Z)-(\nabla_YL_V\nabla)(X,Z).
	\end{align}
	Substituting $\xi$ for $Z$ in the foregoing equation and in view of \eqref{6.10}, we obtain
	\begin{align}
	\nonumber &(L_V R)(X,Y)\xi-(L_V\nabla)(Y,\varphi X)+(L_V\nabla)(X,\varphi Y)\\
	\label{6.11} &=2\{(\nabla_Y Q)\varphi X-(\nabla_X Q)\varphi Y+\eta(Y)QX-\eta(X)QY+(2n-1)(\eta(X)Y-\eta(Y)X)\}.
	\end{align}
	Replacing $Y$ by $\xi$ in \eqref{6.11} and then using last equation of \eqref{2.6}, \eqref{6.7} and \eqref{6.8}, we have
	\begin{align}
	\label{6.12} (L_V R)(X,\xi)\xi=4\{QX-\eta(X)\xi-(2n-1)X\}.
	\end{align}
	Since $Ric^*(X,\xi)=0$, then $*$-Ricci soliton equation \eqref{R2} gives $(L_V g)(X,\xi)+2\lambda\eta(X)=0$, which gives
	\begin{align}
	\label{6.13} (L_V\eta)(X)-g(X,L_V\xi)+2\lambda\eta(X)=0.
	\end{align}
	Lie-derivative of $g(\xi,\xi)=1$ along $V$ gives $\eta(L_V \xi)=\lambda$. Next, Lie-differentiating the formula $R(X,\xi)\xi=X-\eta(X)\xi$ along $V$ and then by virtue of last equation, we obtain
	\begin{align}
	\label{6.14} (L_V R)(X,\xi)\xi-g(X,L_V\xi)\xi+2\lambda X=-((L_V\eta)X)\xi.
	\end{align}
	Combining \eqref{6.14} with \eqref{6.12}, and making use of \eqref{6.13}, we obtain part (ii). This completes the proof.
\end{proof}
By virtue of \eqref{2.9} and \eqref{6.1}, the $*$-Ricci soliton equation \eqref{R2} takes the form
\begin{align}
\label{617}(L_V g)(X,Y)=-\lambda\{g(X,Y)+\eta(X)\eta(Y)\}.
\end{align}
Now, differentiating this equation covariantly along an arbitrary vector field $Z$ on $M$, we have
\begin{align}
\label{6.17} (\nabla_Z L_V g)(X,Y)=-\lambda\{g(Z,\varphi X)\eta(Y)+g(Z,\varphi Y)\eta(X)\}.
\end{align}
Substitute this equation in commutation formula \eqref{6.3}, we find
\begin{align}
\label{6.18} (L_V\nabla)(X,Y)=\lambda\{\eta(Y)\varphi X+\eta(X)\varphi Y\}.
\end{align}
Taking covariant differentiation of \eqref{6.18} along a vector field $Z$ and then using \eqref{2.3}, we obtain
\begin{align}
\nonumber (\nabla_Z L_V\nabla)(X,Y)=&\lambda\{g(Z,\varphi Y)\varphi X+g(Z,\varphi X)\varphi Y+g(X,Z)\eta(Y)\xi\\
\label{6.19} &+g(Y,Z)\eta(X)\xi-2\eta(X)\eta(Y)Z\}.
\end{align}
Making use of \eqref{6.19} in commutation formula \eqref{c1}, we have
\begin{align}
\nonumber (L_V R)(X,Y)Z=&\lambda\{g(X,\varphi Z)\varphi Y+2g(X,\varphi Y)\varphi Z-g(Y,\varphi Z)\varphi X+g(X,Z)\eta(Y)\xi\\
\label{6.21} &-g(Y,Z)\eta(X)\xi+2\eta(X)\eta(Z)Y-2\eta(Y)\eta(Z)X\}.
\end{align}
Contracting \eqref{6.21} over $Z$, we have
\begin{align}
\label{6.22} (L_V Ric)(Y,Z)=2\lambda\{g(Y,Z)-(2n+1)\eta(Y)\eta(Z)\}.
\end{align}
On other hand, taking Lie-differentiation of \eqref{6.1} along the vector field $V$ and using \eqref{617}, we obtain
\begin{align}
\nonumber (L_V Ric)(Y,Z)=&\left[1+\frac{\lambda}{2}\right]\{(L_V\eta)(Y)\eta(Z)+
\eta(Y)(L_V\eta)(Z)\}\\
\label{6.23} &-\lambda\left[2n-1-\frac{\lambda}{2}\right]\{g(Y,Z)+\eta(Y)\eta(Z)\}.
\end{align} 
Comparison of \eqref{6.22} and \eqref{6.23} gives
\begin{align}
\nonumber &\left[1+\frac{\lambda}{2}\right]\{(L_V\eta)(Y)\eta(Z)+
\eta(Y)(L_V\eta)(Z)\} -\lambda\left[2n-1-\frac{\lambda}{2}\right]\{g(Y,Z)+\eta(Y)\eta(Z)\}\\
\label{6.24} &=2\lambda\{g(Y,Z)-(2n+1)\eta(Y)\eta(Z)\}.
\end{align}
Taking $Y$ by $\xi$ in the foregoing equation, we get
\begin{align}
\label{6.25} \left[1+\frac{\lambda}{2}\right](L_V\eta)(Y)=-\left[\lambda+\frac{\lambda^2}{2}\right]\eta(Y).
\end{align}
Substitute \eqref{6.25} in \eqref{6.24} and then replacing $Z$ by $\varphi Z$, we obtain
\begin{align}
\lambda\left[2n+1-\frac{\lambda}{2}\right]g(Y,\varphi Z)=0.
\end{align}
Since $g(Y,\varphi Z)$ is non-vanishing everywhere on $M$, thus we have either $\lambda=0$, or $\lambda=2(2n+1)$.\\
\textbf{Case I:} If $\lambda=0$, then from \eqref{617} we can see that $L_V g=0$, i.e., $V$ is Killing. From \eqref{6.1}, we have
\begin{align*}
Ric(X,Y)=(2n-1)g(X,Y)+\eta(X)\eta(Y).
\end{align*}
This shows that $M$ is $\eta$-Einstein manifold with scalar curvature $r=2n(\alpha+1)=4n^2$.\\
\textbf{Case II:} If $\lambda=2(2n+1)$, then plugging $Y$ by $\varphi Y$ in \eqref{6.25} we have the relation $\left[1+\frac{\lambda}{2}\right](L_V\eta)(\varphi Y)=0$. Since $\lambda=2(2n+1)$, by virtue of last equation we have $\lambda\neq-2$, thus we must have $(L_V\eta)(\varphi Y)=0$. Replacing $Y$ by $\varphi Y$ in the foregoing equation and then using \eqref{2.1}, we have
\begin{align}
\label{6.27} (L_V\eta)(Y)=-2(2n+1)\eta(Y).
\end{align}
This shows that $V$ is a non-strict infinitesimal contact transformation. Now, substituting $Z$ by $\xi$ in \eqref{617} and using \eqref{6.27} we immediately get $L_V\xi=2(2n+1)\xi$. Using this in the commutation formula  (see \cite{YK}, page 23)
\begin{align*}
L_V\nabla_X\xi-\nabla_XL_V\xi-\nabla_{[V,X]}\xi=(L_V\nabla)(X,\xi),
\end{align*}
for an arbitrary vector field $X$ on $M$ and in view of \eqref{a2.4} and \eqref{6.18} gives $L_V\varphi=0$. Thus, the vector field $V$ leaves the structure tensor $\varphi$ invariant.
\par Other hand, using $\lambda=2(2n+1)$ in \eqref{6.1} if follows that
\begin{align*}
Ric(X,Y)=-2g(X,Y)+2(n+1)\eta(X)\eta(Y),
\end{align*}
showing that $M$ is $\eta$-Einstein with $\alpha=-2$. Thus $M$ is a  $D$-homothetically fixed. In\cite{Cha}, the authors give a wonderful information on $\eta$-Einstein Sasakian geometry; in this paper authors says, when $M$ is null (transverse Calabi-Yau) then always $(\alpha,\gamma)=(-2, 2n+2)$ (see page 189). By this we conclude that $M$ is a $D$-homothetically fixed null $\eta$-Einstein manifold. Therefore, we have the following;
\begin{Th}
	Let $M$ be a (2n+1)-dimensional Sasakian manifold. If $M$ admits $*$-Ricci soliton, then either $V$ is Killing, or	$M$ is $D$-homothetically fixed null $\eta$-Einstein manifold. In the first case, $M$ is $\eta$-Einstein manifold of constant scalar curvature $r=2n(\alpha+1)=4n^2$ and in second case, $V$ is a non-strict infinitesimal contact transformation and leaves the structure tensor $\varphi$ invariant. 
\end{Th}
Now, we prove the following result, which gives some remark on $*$-Ricci soliton.
\begin{Th}
	Let $M$ be a (2n+1)-dimensional Sasakian manifold admitting $*$-Ricci soliton with $Q^*\varphi=\varphi Q^*$. Then the soliton vector field $V$ leaves the structure tensor $\varphi$ invariant if and only if $g(\varphi(\nabla_V \varphi)X,Y)=(dv)(X,Y)-(dv)(\varphi X,\varphi Y)-(dv)(X,\xi)\eta(Y)$. 
\end{Th}
\begin{proof}
	The $*$-Ricci soliton equation can be written as 
	\begin{align}
	\label{6.28} g(\nabla_XV,Y)+g(\nabla_YV,X)+2Ric^*(X,Y)+2\lambda g(X,Y)=0.
	\end{align}	
	Suppose $v$ is 1-form, metrically equivalent to $V$ and is given by $v(X) = g(X, V)$, for any arbitrary vector field $X$, then the exterior derivative $dv$ of $v$ is given by
	\begin{align}
	\label{6.29} 2(dv)(X,Y)=g(\nabla_X V,Y)-g(\nabla_YV,X)
	\end{align}
	As $dv$ is a skew-symmetric, if we define a tensor fieeld $F$ of type (1, 1) by
	\begin{align}
	(dv)(X,Y)=g(X,FY),
	\end{align}
	then $F$ is skew self-adjoint i.e., $g(X, FY)=-g(FX, Y)$. The equation \eqref{6.29} takes the form $2g(X,FY)=g(\nabla_X V,Y)-g(\nabla_YV,X)$. Adding it to equation \eqref{6.28} side by side and factoring out $Y$ gives
	\begin{align}
	\label{6.31}\nabla_XV=-Q^*X-\lambda X-FX,
	\end{align}
	where $Q^*$ is $*$-Ricci operator. Applying $\varphi$ on \eqref{6.31}, we have
	\begin{align}
	\label{6.32} \varphi\nabla_XV=-\varphi Q^*X-\varphi\lambda X-\varphi FX.
	\end{align}
	Next, Replacing $X$ by $\varphi X$ in \eqref{6.31}, we obtain
	\begin{align}
	\label{6.33} \nabla_{\varphi X}V=-Q^*\varphi X-\lambda\varphi X-F\varphi X.
	\end{align}
	Substracting \eqref{6.32} and \eqref{6.33}, we have
	\begin{align}
	\varphi\nabla_X V-\nabla_{\varphi X}V=(Q^*\varphi-\varphi Q^*)X+(F\varphi-\varphi F)X.
	\end{align}
	By our hypothesis, noting that $\varphi$ commutes with the $*$-Ricci operator $Q^*$ for Sasakian manifold, we have
	\begin{align}
	\label{6.35}\varphi\nabla_X V-\nabla_{\varphi X}V=(F\varphi-\varphi F)X.
	\end{align}
	Now, we note that
	\begin{align*}
	(L_V\varphi)X&=L_V\varphi X-\varphi L_VX\\
	&=\nabla_V\varphi X-\nabla_{\varphi X}V-\varphi\nabla_VX+\varphi\nabla_XV\\
	&=(\nabla_V\varphi)X-\nabla_{\varphi X}V+\varphi\nabla_XV.
	\end{align*}
	The use of foregoing equation in \eqref{6.35} gives
	\begin{align}
	\label{6.36}(L_V\varphi)X-(\nabla_V\varphi)X=(F\varphi-\varphi F)X.
	\end{align}
	Operating $\varphi$ on both sides of the equation \eqref{6.35} and then making use of \eqref{2.1}, \eqref{6.31} and \eqref{6.33}, we find
	\begin{align*}
	(dv)(\varphi X,\varphi Y)-(dv)(X,Y)+(dv)(X,\xi)\eta(Y)=g(\varphi(F\varphi-\varphi F)X,Y).
	\end{align*}
	Using \eqref{6.36} in the above equation provides
	\begin{align*}
	(dv)(\varphi X,\varphi Y)-(dv)(X,Y)+(dv)(X,\xi)\eta(Y)=g(\varphi(L_V\varphi)X-\varphi(\nabla_V\varphi)X, Y).
	\end{align*}
	This shows that $L_V\varphi=0$ if and only if  $g(\varphi(\nabla_V \varphi)X,Y)=(dv)(X,Y)-(dv)(\varphi X,\varphi Y)-(dv)(X,\xi)\eta(Y)$, completing the proof.
\end{proof}


\begin{thebibliography}{}
	
	\bibitem{Ar} Arslan, K., De, U.C., Ozgur, C., Jun, J.B.: {\em On class of Sasakian manifolds.} Nihonkai Math. J. 19, 21-27 (2008)	
	
	\bibitem{DEB} Blair, D.E.: {\em Contact Manifolds in Riemannian Geometry.}\,\, Lecture Notes in Mathematics 509, Springer-Verlag, Berlin-New York, 1976.
	
	\bibitem{Bo} Boyer, C.P., Galicki, K.: {\em On Sasakian-Einstein geometry.} Internat. J. Math. 11, 873-909 (2000)
	
	\bibitem{CAR} Cabrerizo, J.L., Fern\'{a}ndez, L. M., Fern\'{a}ndez, M.,  Zhen, G.: {\em The structure of a class of K-contact manifolds.} Acta Math. Hungar. 82 \textbf{4}, 331-340 (1999)
	
	\bibitem{Chaki} Chaki, M.C.,  Tarafdar, M.: {\em On a type of Sasakian manifold.} Soochow J. Math. 16 (1), 23-28 (1990)
	
	\bibitem{Cha} Charles, P.B., Krzysztof, G., Paola, M.: {\em On $\eta$-Einstein Sasakian geometry.} Commun. Math. Phys. 262, 177-208 (2006)
	
	\bibitem{JTC} Cho, J.T.,  Inoguchi, J.I.: {\em On $\varphi$-Einstein contact Riemannian manifolds.} Mediterr. J. Math. 7, 143-167 (2010)
	
	\bibitem{De} De, U.C., Ozgur, C., Mondal, A.K.: {\em On $\phi$-quasiconformally symmetric Sasakian manifolds.} Indag. Mathern., N.S. 20 (2), 191-200 (2009)
	
	\bibitem{Gei} Geiges, H.: {\em A brief history of contact geometry and topology.} Expo. Math. 19 (1), 25-53 (2001)
	
	\bibitem{Geo}George, K., Konstantina, P.: {\em $*$-Ricci solitons of real hypersurfaces in non-flat complex space forms.} J. Geom. Phys. 86, 408-413 (2014)
	
	\bibitem{AG} Ghosh, A., Sharma, R.: {\em Sasakian metric as a Ricci soliton
		and related results.} J. Geom. Phys. 75, 1-6 (2014)
	
	\bibitem{HT} Hamada, T., Inoguchi, J.I.: {\em Real hypersurfaces of complex space forms with symmetric Ricci $*$-tensor.} Mem. Fac. Sci. Eng. Shimane Univ. Series B: Mathematical Science. 38,  1-5 (2005)
	
	\bibitem{HRS} Hamilton, R.S.:\,\ {\em The Ricci flow on surfaces.}\,\,  Contemporary Mathematics. 71, 237-261 (1988)
	
	\bibitem{Ikw} Ikawa T., Kon M.: {\em Sasakian manifolds with vanishing contact Bochner curvature tensor and constant scalar curvature.} Colloq. Math. 37 (1), 113-122 (1977)
	
	\bibitem{JB} Jun, J.B.,  Kim, U.K.: {\em On 3-dimensional almost contact metric manifolds.} Kyungpook Math. J. 34 (2), 293-301 (1994)
	
	\bibitem{Kus} Kushwaha, A., Narain, D.: {\em Some curvature properties on Sasakian manifolds.} J. Int. Aca. Phy. Sci.  20 (4), 293-302 (2016)   
	
	\bibitem{KON} Kon, M.: {\em Invariant submanifolds in Sasakian manifolds.} Math. Ann. 219 (3), 277-290  (1976)
	
	\bibitem{OKU}  Okumura, M.: {\em Some remarks on space with a certain contact structure.} Tohoku Math. J.  14, 135-145 (1962)
	
	\bibitem{Olz} Olszak, Z.: {\em Certain property of the Ricci tensor on Sasakian manifolds.} Collo. Math. 40 (2), 235-237 (1978/79)
	
	\bibitem{Prak} Prakasha, D.G., Veeresha, P.: {\em Para-Sasakian manifolds and $*$-Ricci soliton.} arXiv:1801.01727v1 (2018)
	
	\bibitem{Sasa} Sasaki, S.: {\em Lecture notes on almost contact manifolds, Part I.} Tohoku Univ. 1965
	
	\bibitem{TS} Tachibana, S.: {\em On almost-analytic vectors in almost Kahlerian manifolds.} Tohoku Math. J. 11, 247-265 (1959)
	
	\bibitem{Tanno} Tanno, S.: {\em Isometric immersions of Sasakian manifolds in spheres.} K$\bar{o}$dai Math. Sem. Rep. 21, 448-458 (1969)
	
	\bibitem{YK} Yano, K., {\em Integral formulas in Riemannian geometry.} Marcel Dekker. New York, (1970)
	
	\bibitem{YKK} Yano, K., Kon, M.: {\em Structure on manifold.} world scientific publishing, Series in pure mathematics. 3 (1984).
	
	
	
\end{thebibliography}
\end{document}